\documentclass[a4paper,english,oneside]{amsart}
\usepackage[utf8]{inputenc}
\usepackage[T1]{fontenc}

\usepackage[hscale=0.6, vscale=0.66]{geometry}

\usepackage{amsmath,amssymb,amsthm}

\DeclareFontFamily{OMX}{mlmex}{}
\DeclareFontShape{OMX}{mlmex}{m}{n}{%
   <->mlmex10%
   }{}%
\usepackage{mlmodern}

\theoremstyle{plain}
\newtheorem{theorem}{Theorem}
\newtheorem{proposition}[theorem]{Proposition}
\newtheorem{lemma}[theorem]{Lemma}

\theoremstyle{definition}

\newtheorem{remark}[theorem]{Remark}

\usepackage{color}

\usepackage{hyperref}
\hypersetup{%
  colorlinks=true,%
  citecolor=[RGB]{120,29,126},
  pdfauthor={Jean-François Burnol},%
  pdfsubject={Kempner series},%
  pdfstartview=FitH,%
  pdfpagemode=UseNone,%
}
\usepackage{breakurl}

\DeclareMathOperator{\ZZ}{\mathbb{Z}}
\DeclareMathOperator{\RR}{\mathbb{R}}

\usepackage[numbers,sort&compress]{natbib}

\title[Moments for Kempner series]{Moments in the exact summation of the\\%
curious series of Kempner type}

\author[J.-F. Burnol]{Jean-François Burnol}

\address{Université de Lille, Faculté des Sciences et technologies,
  Département de mathématiques, Cité Scientifique, F-59655 Villeneuve d'Ascq
  cedex, France}
\email{jean-francois.burnol@univ-lille.fr}
\date{October 2024. To appear (with better introduction, conclusion and style) in American
  Mathematical Monthly in 2025. Compared to v1, the definition of set curly A was
  slightly changed to handle digit zero allowed or not in a
  more unified way (the measure is exactly same as in v1).  References to
  further works by the author and others got added.}
\subjclass[2020]{Primary 11Y60; Secondary 11A63; 44A60; 40-04}
\keywords{Kempner series, numerical evaluation of mathematical constants}

\usepackage{setspace}
\begin{document}

\singlespacing
\begin{abstract}
  We obtain for the Kempner series (i.e.\@ harmonic series where certain
  digits are excluded from all denominators, for example the digit 9 in base
  10) new representations as geometrically convergent series.  The
  coefficients for these representations involve the power sums on the allowed
  digits and the moments of an explicitly described measure on the unit
  interval.  These moments can be computed numerically by recurrence.  We
  establish a priori lower and upper bounds for them.  This allows converting
  the theoretical formulas into an efficient numerical algorithm.
\end{abstract}

\maketitle

\onehalfspacing
\section{Introduction.}

Let $b>1$ an integer.  The $b$-ary representation of integers uses the
elements of the set $\{0,\dots, b-1\}$ as digits.  A positive integer has a
unique minimal representation as a string of such digits, the leftmost one
(also named the ``leading'', or ``most significant'' digit) being non-zero.

Kempner \cite{kempner} and Irwin \cite{irwin}
(both working exclusively with $b=10$) considered infinite subsums of the
harmonic series with constraints on the digits in the denominators.  Kempner
allowed only those terms having no occurrence of a given digit $d$ and proved
the convergence.  Irwin proved the
convergence under conditions of the type ``exactly three occurrences of $5$ or
at most one occurrence of $2$''.

Kempner \cite{kempner} (and also Irwin \cite{irwin}) called the convergence
``curious''.
The Kempner result is Theorem 144 in the
widely-read book  \cite{hardywright} by Hardy and Wright.

Baillie \cite{baillie1979,baillie2008} provided numerical algorithms to
evaluate Kempner and Irwin series.  Schmelzer and Baillie
\cite{schmelzerbaillie} extended the Kempner ``no occurrence'' constraint from
a single digit to consecutive digits.  They obtain (see also
\cite{baillie2008}) a numerical algorithm (with implementation) to compute
such infinite sums.

Recently Allouche and Morin \cite{allouchemorin2023} for
$b=2$, and Allouche, Hu and Morin
\cite{allouchehumorin2024} for general $b>1$ obtained a limit theorem for the
harmonic series where denominators have exactly $k$ occurrences of a given
fixed string of $p$ digits: they tend to $b^p\log(b)$ as $k\to\infty$.  The
$p=1$ case had been known since the work of Farhi \cite{farhi}.  Many
generalizations of Kempner-like conditions have been studied \cite{craven1965,
  alexander1971, fischer, behforooz1995, wadhwa1975,wadhwa1978,
  kohlerspilker2009, lubeckponomarenko2018, gordon2019, mukherjeesarkar2021,
  nathanson2021integers, nathanson2021jnt, nathanson2022ramanujan}. We comment
more in the concluding section on Fischer's work \cite{fischer} as it bears
some close relation to the approach of the present paper.

Lubeck and Ponomarenko \cite{lubeckponomarenko2018} (see also
\cite{gordon2019} by Gordon) introduce (in the
second part of their paper) a probabilistic flavor by considering an integer
$n$ having $q$ digits in base $10$ as the result of $q$ random choices, with
the constraint that the first digit must be positive.  For a fixed parameter
$\lambda$ they consider the series of reciprocals of the positive integers
such that the proportion of $9$'s among their decimal digits is at most equal
to $\lambda$.  For $\lambda\geq\frac1{10}$ probabilistic intuition suggests
that the condition will keep among the denominators ``typical'' numbers (or,
for $\lambda=\frac1{10}$, ``half'' of the ``typical'' numbers), hence that the
series diverges.  For $\lambda<\frac1{10}$ probabilistic intuition suggests
that very few denominators, among those sharing the same length $N$, $N\gg1$,
are kept and that the series converges.  Using tools such as the
Chernoff bound
(see their paper for statement and reference) on the lower tail of a binomial
or Poisson distribution, and Riemann-Stieltjes integration, they validate both
statements.  Measures and integration will also provide the framework of this
paper.

One challenge of the Kempner series is
that evaluating them numerically, even at the computer age, is a non-trivial task.
Baillie \cite{baillie1979} obtained, to twenty decimal places each, the ten sums
where the excluded decimal digit ranges from $0$ through $9$.  Let us briefly
describe, with $d$ being the excluded decimal digit, initial steps leading to
Baillie's algorithm.  Let $S_l=\sum'_{l(n)=l} n^{-1}$ be the $l$-th
``block'' gathering contributions from all allowed denominators having
$l(n)=l$ digits.  The primed summation \smash{$\sum'$} is notation for the
restriction to only those terms obeying the ``no $d$'' condition.  Using the
upper and lower bounds provided by the expansion
$1/(10x+a) = 1/(10x) - a/(100x^2)+\dots$, and defining $S_{l,2}$ as the
variant of $S_{l}$ where all terms are squared, we obtain
$0.9\,S_l - \frac{45-d}{100}\,S_{l,2}< S_{l+1}<0.9\,S_l$, where $45-d$ is the
sum of the $a\neq d$. Using more terms of the expansion of
$(10x+a)^{-1}$, where $a\neq d$ and $x$ already satisfies the ``no $d$''
condition, one is led to define generally $S_{l,k}=\sum'_{l(n)=l} n^{-k}$ and
to express $S_{l+1,k}$ as an infinite combination of the $S_{l,k'}$,
$k'=k,k+1,\dots$.  After computing a large enough number of $S_{l,k}$'s for
small $l$'s and many $k$'s one can propagate numerically to higher $l$'s.
This idea requires careful formalization and book-keeping, and considerations
going beyond the brief summary we have given.  As our present work does not
use this directly, we refer the reader to \cite{baillie1979} and
\cite{baillie2008}, as well as \cite{schmelzerbaillie} for the actual details.

We obtain a theoretically exact representation of Kempner sums as alternating
series with terms associated to certain recurrences (Theorem
\ref{thm:main}). The a priori estimates given in Theorem \ref{thm:main}, and
the alternating nature of the series it considers, allow to convert its
formulas directly into an efficient numerical algorithm: two
\textsc{Maple}\texttrademark{} scripts are available as ``ancillary files'' at
\url{https://arxiv.org/abs/2402.08525} (they handle $b=10$).  \textsc{Python}
and \textsc{SageMath} scripts (for any $b>1$) are available at
\url{https://arxiv.org/abs/2402.09083}: they handle the case of harmonic
series with denominators constrained to have a given number $k$ of occurrences
of a single $b$-ary digit $d$.  For $k=0$, they implement the formulas of Theorem
\ref{thm:main}.

\section{Notations and main result.}

Let $A$ be a subset of the set $\{d\in \ZZ, 0\leq d<b\}$ of $b$-ary digits.  Its
elements will be called the ``admissible'' (or, sometimes, the ``allowed'') digits.  For reasons explained next, we
exclude the cases of $A$ either containing all digits, or being reduced to
$\{0\}$, or being empty.  Let
\begin{equation}
  K(A) = \sum_{n>0}\nolimits' \frac1n\quad,
\end{equation}
where the prime symbol $'$ restricts the summation to those denominators $n$
having all their digits from the set $A$.
With $A$ being all of $\{0,\dots,b-1\}$, this definition would give the full
harmonic series, which diverges, and we exclude this case for that reason.  For
$A$ empty or $\{0\}$, the series would be empty, and its numerical value
$K(A)$ would be zero.  Although a number of constructions do work, by
excluding also these two latter cases we avoid repeated special provisions in the
statements and proofs thus providing a smoother exposition.

We call a positive integer admissible if its $b$-ary representation contains
only admissible digits.  The integer $0$ is always considered admissible (its
minimal $b$-ary representation is the empty word hence verifies the condition
of having only digits from $A$).  So $0$ as integer is always admissible but
$0$ as a digit may not be.  This dual nature of $0$ will hopefully not cause
confusion.  We let $\mathcal{A}$ be the set of admissible non-negative
integers.  In the sequel, the primed summation symbol $\sum'$ means that only
integers in $\mathcal{A}$ are included.

Let $N=\# A$ and $N_1 = \#(A \setminus\{0\})$.  The contribution to $K(A)$
from integers with $b^{l-1}\leq n< b^l$ is bounded above by $N_1\cdot
N^{l-1}/b^{(l-1)}$, hence the series converges as $N<b$. We next have to
evaluate it numerically\dots

Grouping integers according to their number of $b$-ary digits is a fundamental
tool ever since Kempner's original work and Baillie's algorithms
\cite{baillie1979,baillie2008}.  The notation $l(n)$, already used earlier
with decimal representations, will denote for $n>0$ the number of its $b$-ary
digits (the integer $b>1$ being fixed throughout this text).  Thus
$l(n)$ is the smallest non-negative exponent such that $n<b^l$.  Applying this
definition also to $n=0$ leads us to decide that $l(0)=0$.  This aligns
with $0$ having as minimal representation the empty word.

To state the main theorem we define quantities $\beta_{l,m}$ for $l\geq1$,
$m\geq0$:
\begin{equation}\label{eq:beta}
  \beta_{l,m} =\sum_{b^{l-1}\leq n < b^{l}}\nolimits'  n^{-m}
              = \sum_{l(n)=l}\nolimits'  n^{-m} \;,
\end{equation}
and
\begin{equation}\label{eq:gamma}
  \gamma_j =
  \begin{cases}
    \sum_{a\in A} a^j = \sum_{l(n)=1}\nolimits' n^j&(j\geq1)\;,
\\
    N &(j=0)\;.
  \end{cases}
\end{equation}
Note that the right-most expression given above for $j\geq1$ would, if applied
to $j=0$, compute $N_1 = \#(A \setminus\{0\})$ which is either $N=\#A$ or $N-1$.
\begin{theorem}\label{thm:main}
  Let $\ell\geq1$ be arbitrarily chosen.  Let $K = \sum'_{n>0}
  n^{-1}$ where the positive integers $n$ are those having all their $b$-ary
  digits from a set $A$
  of cardinality $N$, $0<N<b$, $A\neq\{0\}$. 

\noindent\textup{(a)} The ``Kempner sum'' $K$ can be evaluated using:
  \begin{equation}\label{eq:main}
    K = \sum\nolimits'_{0<n<b^{\ell-1}} \frac1n +
       \frac{b}{b-N}\sum\nolimits'_{b^{\ell-1}\leq n<b^\ell} \frac1n  +
       \sum_{m=1}^\infty (-1)^m u_m \beta_{\ell,m+1}\;.
  \end{equation}
  The quantities $\beta_{\ell,m}$'s are those
  defined in equation \eqref{eq:beta}.
  The $u_m$'s, which are independent of $\ell$, are uniquely determined by the conditions 
  $u_0=b/(b-N)$ and
  \begin{equation}\label{eq:recur}
    (b^{m+1} - N) u_m = \sum_{j=1}^{m} \binom{m}{j} \gamma_j u_{m-j},\quad (m\geq1).
  \end{equation}
  The $\gamma_j$'s used in this recursive definition are those from \eqref{eq:gamma}.

\noindent\textup{(b)}
  The quantities $\lambda_{m+1}$
  implicitly defined by
  \begin{equation}
    u_m = \frac{\lambda_{m+1}}{m+1}\biggl(\frac{\max A}{b-1}\biggr)^{\!m}\frac{b}{b-N}\quad,
  \end{equation}
  satisfy $\lambda_1=1$ and $b^{-1} < \lambda_{m+1} < 1$ for $m+1\geq2$.  The
  ratios $\lambda_{m+1}/(m+1)$ are positive and decreasing towards
  zero.

\noindent\textup{(c)}
  The convergence of \eqref{eq:main} is bounded geometrically except when simultaneously
  $\ell=1$,
  $1\in A$, and $b-1 \in A$.
\end{theorem}
The proof will be the object of the rest of this text.  For now, we make a few
simple comments.

For the original ``no $9$'' Kempner series, using Theorem \ref{thm:main} with $\ell=1$,
the first upper bound is $10(\frac11 + \dots + \frac18)\approx 27.18$. With
$\ell=2$ it is $\frac11 + \dots + \frac18 + 10 (\frac1{10} + \dots +
\frac1{18} + \frac1{20} + \dots + \frac1{88}) \approx 23.26$ which is already
not so far from $K = 22.9206766\dots$.  The $\ell=3$ initial approximant turns out
to be about $22.95$.

With $\ell=1$ and still for the ``no $9$'' $K$, the $\beta_{1,m+1}$'s are
$O(1)$ (they are the sum of $1^{-m-1}$ to $8^{-m-1}$) and $u_m \asymp
(8/9)^m/m$, so the series from equation \eqref{eq:main} converges at a rate similar to the convergence of the Taylor series for
$\log(1 + \frac89)$, which is not satisfactory numerically.

For $\ell=2$ on the other hand, $\beta_{2,m+1}=O(10^{-m})$ and the $u_m$'s are
the same
as for $\ell=1$, so here a geometric convergence with ratio $8/90$ is
guaranteed, meaning (ignoring implied constants) roughly one more decimal
digit of the final sum at each added term.  For $\ell=3$ we get about two
additional decimal digits of the Kempner sum at each added term.

\section{The Kempner sums as integrals.}

A non-negative measure $\mu$ on the real line $\RR$ with $\mu(\RR)=c>0$ is said to be a
Dirac point mass if there exists some (necessarily unique) $x$ with
$\mu(\{x\}) = c$.  We then use the notation $\mu = c\delta_x$, and call
$\delta_x$ the ``Dirac distribution at $x$'' (with unit mass).

Given an enumerated subset $\{x_1,x_2,\dots\}$ of the real line, and a series with
non-negative terms $\sum_{n\geq1} c_n$, (possibly diverging), one can define
the set-function $\mu:\mathcal{P}(\RR)\to \RR_{\geq0}\cup\{+\infty\}$ which
assigns to any subset $G$ of the real line the quantity $\mu(G) =
\sum_{n=1}^\infty c_n\mathbf{1}_{G}(x_n)\in[0,\infty]$.  This set-function is
countably additive and we call it a (non-negative, discrete, possibly
infinite) measure.  We define integrable functions as being those functions
with $\sum_{n=1}^\infty c_n |f(x_n)|<\infty$.  Then $\int_{\RR} f(x)\,d\mu(x)$
is defined as expected and obeys the usual theorems for integration (linearity, monotony,
Lebesgue limit theorems) and also is invariant under any rearrangement of the
indexing of the set $\{x_i\}$.  We write $\mu = \sum_i c_i\delta_{x_i}$,
where the countable index set to which $i$ belongs does not have to be
$\ZZ_{>0}$.

When such an infinite discrete weighted sum of Dirac masses is locally finite
(i.e.\@ assigns a finite mass to any finite length interval) it can be used to
define a linear functional on compactly supported continuous functions and may
be seen as defining also a
Schwartz distribution.  This will be the case for the measure we define next
in association to a base $b$ and a (proper) set $A$ of admissible digits.  But
we do not need any real measure theory here.  Our integrated functions do not
need to verify any continuity conditions.  We employ the language as an
alternative to and convenient notation for direct re-arrangements of various
summable series.

Let the measure
$\mu_{b,A}$ on $[0,+\infty)$ be defined this way:
\begin{equation}
  \mu_{b,A} = \begin{cases}
    \sum_{l\geq0} b^{-l}\sum_{n\in\mathcal{A}} \delta_{n/b^{l}} & \text{if }0\in A\;,\\
    \sum_{l\geq0} b^{-l}\sum_{n\in\mathcal{A}, l(n)\geq l} \delta_{n/b^{l}} &
    \text{if }0\notin A\;.
  \end{cases}
\end{equation}
As $b$ and $A$ will be fixed throughout the paper, we drop the indices and use
simply $\mu$ as notation. Also from here on we shall use $K$ as an abbreviation
for $K(A)$.

In the case $0\notin A$, the condition $l(n)\geq l$ means that either $n=0$
and $l=0$ (recall that $0\in\mathcal{A}$) or $n/b^l\geq b^{-1}$. After
separating the contributions originating from $l=0$ we obtain the following
equivalent representation which we will use preferentially:
\begin{equation}\label{eq:mu0notinA}
  \text{For }0\notin A:\mu = \sum_{n\in\mathcal{A}}\delta_n + 
        \sum_{l\geq1} b^{-l}\sum_{n\in\mathcal{A}, n\geq b^{l-1}}
        \delta_{n/b^{l}}\;.
\end{equation} 

Even before checking that $\mu$ assigns a finite mass to $[0,1)$ and more
generally to all bounded intervals, let us relate immediately $K$ to an integration:
\begin{equation}\label{eq:1}
  \int_{[b^{-1},1)} \frac{d\mu(x)}x =  
   \sum_{l\geq0} \frac1{b^l} \sum_{b^{l-1}\leq n< b^l}\nolimits' \frac{b^l}{n}
                   = \sum_{n>0}\nolimits' \frac1n = K\;.
\end{equation}
Any good approximation of $x^{-1}$ in the uniform norm by a polynomial on the
interval $[b^{-1},1]$ will thus lead to good approximations of $K$, assuming
we know the $v_m = \int_{[b^{-1},1)} x^m\,d\mu(x)$.  In this paper we shall
rather work with the moments of $\mu$ on $[0,1)$:
\begin{equation}\label{eq:moment}
  u_m = \int_{[0,1)} x^m\,d\mu(x)\qquad (m\geq0).
\end{equation}
We first need to check that $u_0$ and, hence, the $u_m$'s are all finite.
Suppose $0\in A$.  Then
\begin{equation*}
  \mu([0,1)) =  \sum_{l\geq0} b^{-l} \#\{n\in \mathcal{A}, n < b^{l}\}
= \sum_{l\geq0} b^{-l} N^{l} = \frac{b}{b-N}\quad.
\end{equation*}
If $0\notin A$, the computation is different but the result identical:
\begin{equation*}
    \mu([0,1)) = 1 + \sum_{l\geq1} b^{-l} \#\{n\in \mathcal{A},
    b^{l-1}\leq n < b^{l}\} = 1 + \sum_{l\geq1} b^{-l}N^l =\frac{b}{b-N}\quad.
\end{equation*}

Here is a more detailed description of the points in the support of $\mu$
and their associated masses:
\begin{itemize}
\item If $0\in A$, a positive rational number $x$ is in the support of $\mu$
  if and only if its irreducible representation as a fraction is $n/b^l$, with
  $n$ an admissible integer and $l\geq0$.  The weight in $\mu$ of $\delta_x$ is
  $b^{-l}+b^{-l-1}+ \dots = b^{-l}(1-1/b)^{-1}$.

  There are points arbitrarily close to $0$ in the support of
  $\mu$.

  The total weight of the Dirac at the origin is $1+ b^{-1}+b^{-2}+\dots =
  b/(b-1)$, it is the same weight as applies to the $\delta_n$'s with $n$
  a positive admissible integer.

\item If $0\notin A$, a positive rational number $x$ is in the support of $\mu$
  if and only if its irreducible representation as a fraction is $n/b^l$, with
  $l\geq0$ and $n$ an admissible positive integer of length at least $l$.  The weight
  of $\delta_x$ for such $x=n/b^l$, $b\nmid n$, $l\geq0$ is $b^{-l}$.

  No points in the open interval $(0,b^{-1})$ are in the support of
  $\mu$.  The Dirac at the origin appears in $\mu$ with weight $1$ as do the
  $\delta_n$'s for $n$ a positive admissible integer.
\end{itemize}
In both cases the only positive integers in the support are the
admissible ones.
If $x$ is in the support and is not an integer, then $bx$ is also
in the support and the weight of $\delta_{bx}$ in $\mu$ is $b$ times the one of
$\delta_x$.
\begin{proposition}
  If $n$ is a positive integer, the restriction of $\mu$ to $[n,n+1)$ is zero
  if $n$ is not admissible, and is the translate by $n$ of the restriction of
  $\mu$ to $[0,1)$ if $n$ is admissible.
\end{proposition}
\begin{proof}
  Let us first suppose that $0\notin A$.  We have already mentioned that a
  positive integer $n$ is in the support if and only if it is admissible, and
  the weight in $\mu$ of $\delta_n$ is then $1$ (it would be $b/(b-1)$ for
  $0\in A$).  So let us examine when $x\in
  (n,n+1)$ is in the support.  Its irreducible representation has to be
   (as $x$ isn't an integer)
  $x=m/b^l$ with $l\geq1$ and $m$ admissible
  having at least $l$ digits. But as $x>1$, $m$ must automatically have $q>l$
  digits.  Then $n$, which is the integral part of $x$, is the integer
  obtained from the first $q-l$ digits of $m$. It is thus admissible.  Let us
  now write $m=b^ln+p$ with $0\leq p<b^l$. Then $p$ is admissible and positive
  as it is
  obtained from the last $l\geq1$ digits of $m$.  In particular it is not divisible
  by $b$ (in the $0\in A$ case we would rather say $b\nmid m\implies b\nmid
  p$) and has to be at least $b^{l-1}$ (no analog in the $0\in A$ case).  So
  $y=x-n=p/b^l \in (0,1)$ is in the support and has the very same weight
  $b^{-l}$ as $x$. Conversely starting from such a $y\in(0,1)$ in the support
  of $\mu$, it is $p/b^l$ for some admissible $p\geq b^{l-1}$ so $m=b^ln+p$ is
  admissible and $x=y+n=m/b^l$ is in the support with weight $b^{-l}$.
  
  The case with $0\in A$ is similar and simpler and is left to the reader.
\end{proof}
Although the previous proposition related $\mu$ ``in the large'' with
$\mathcal{A}$ and perhaps suggests investigating $K$ from this perspective, we stated
it only for completeness as in this paper we will be only using the interval
$[0,1)$.

We already expressed $K$ as an integral of $1/x$.  The following lemma is
crucial to obtain other integral expressions.

For a positive integer $m$ with at least $l$ digits, we let $\mathsf{ld}_l(m)$
be the integer in $[b^{l-1},b^l)$ sharing with $m$ its $l$
``leading digits''.  If $m$ belongs to $\mathcal{A}$, so does
$\mathsf{ld}_l(m)$.
\begin{lemma}\label{lem1}
  For any non-zero $n\in\mathcal{A}$ of length $l(n)$: 
  \begin{equation*}
    \int_{[0,1)} \frac{1}{n + x}\,d\mu(x) =
   \sum\nolimits'_{\mathsf{ld}_{l(n)}(m) = n}\frac1m\;.
  \end{equation*}
\end{lemma}
\begin{proof}
  Let us first suppose $0\in A$.  The integral on the left-hand-side
  is the series
  $\sum_{l\geq0} b^{-l} \sum'_{p<b^l} 1/(n + p/b^l)$ where the prime indicates as
  usual that only
  admissible $p$'s are used.  The $l=0$ contributes only $p=0$ which gives
  $\frac1n$. Admissible integers $m$
  with $\mathsf{ld}_{l(n)}{m} = n$ are, if not equal to $n$ itself, in
  one-to-one correspondance with pairs $(l,p)$ with $l\geq1$, and $p<b^l$ admissible.
  The correspondance goes via the equality $m = n b^l + p$.

  If $0\notin A$, the left-hand-side is the sum of $\frac1n$ (which comes from
  $\delta_0$ in equation \eqref{eq:mu0notinA}) with $\sum_{l\geq1} b^{-l}
  \sum'_{b^{l-1}\leq p<b^l} 1/(n + p/b^l)$ (where the prime symbol means that
  $p$'s are admissible).  The integer $b^l n + p$ (with $n>0$, $0\leq p<b^l$)
  is admissible if and only if $p$ is admissible and verifies $p\geq b^{l-1}$.
  Hence the result.
\end{proof}

\section{The Kempner sums as alternating series.}

\begin{theorem}\label{thm:1}
  For all $\ell\geq1$ the Kempner-like series has value
  \begin{equation}\label{eq:main0}
    K = \sum_{0<n<b^{\ell-1}}\nolimits'\frac1n + \sum_{m=0}^\infty (-1)^m u_m 
                           \sum_{l(n)=\ell}\nolimits' \frac1{n^{m+1}}\;.
  \end{equation}
  The sequence of moments  $(u_m)$ is
  decreasing and converges to zero.  Thus, the alternating series has its terms
  decreasing towards zero in absolute value.
\end{theorem}
\begin{proof}
  The fact that $u_m>u_{m+1}$ is clear from the integral representation
  \eqref{eq:moment}.  The
  convergence to zero also follows from \eqref{eq:moment} using either
  dominated convergence, or via the following elementary argument: $\forall\epsilon\in(0,1)$,
  $\int_{[0,1)} x^m\,d\mu(x)\leq (1 - \epsilon)^m u_0 +
  \mu([1-\epsilon,1))$, so the limit superior is at most
  $\mu([1-\epsilon,1))$ for any $\epsilon>0$, hence is zero.

  As a first step towards \eqref{eq:main0}, we gather all contributions
  $n^{-1}$ from those
  $n$'s which are $\geq b^{\ell-1}$ and share a common $\mathsf{ld}_l(n)$.
  Lemma \ref{lem1} then gives the following expression:
\begin{equation}
  \forall \ell\geq1\quad K = \sum_{0<n<b^{\ell-1}}\nolimits'\frac1n +
      \sum_{l(n)=\ell}\nolimits' \int_{[0,1)} \frac{1}{n + x}\,d\mu(x)\;.
\end{equation}
Even for  $n=1$ we can always expand:
\begin{equation*}
  \int_{[0,1)} \frac{1}{n + x}\,d\mu(x) = 
  \frac{u_0}{n} - \frac{u_1}{n^2} + \frac{u_2}{n^3} + \dots\quad.
\end{equation*}
Indeed, $1/(n+x) = 1/n - x/n^2 + x^2/n^3 - \dots$ is an alternating series
whose partial sums provide alternatively upper and lower bounds for the
integrand $(n+x)^{-1}$.  Hence $\int_{[0,1)} \frac{1}{n + x}\,d\mu(x)$ is
alternatively bounded above and below by the partial sums of the series $\sum
(-1)^m u_m/n^{m+1}$, depending on the parity of the number of kept terms.  The
latter series converges, also for $n=1$, as $u_m$ decreases to zero.  Hence
the above displayed formula holds, also for $n=1$.  For $n>1$, the more usual argument is to
mention the uniform convergence on $[0,1]$ of the Taylor series for
$x\mapsto(n+x)^{-1}$.

Summing over all admissible integers $n$ of exact length $\ell$ we obtain the
formula \eqref{eq:main0}.
\end{proof}
\begin{remark}
  The contribution of $m=0$ in \eqref{eq:main0} is the product of $u_0= \frac
  b{b-N}$ by the sum of the reciprocals of the admissible integers of length
  $\ell$.  It is natural to separate it from the rest of the series and
  combine it with the contribution of the numbers with fewer digits.  See the
  statement of Theorem \ref{thm:main}.  We say that the formula is ``at level
  $\ell$''.
\end{remark}
In the rest of the paper we explain how to compute the coefficients $u_m$ so
as to transform Theorem \ref{thm:1} into a numerical algorithm.  We also
evaluate theoretically their order of magnitude which allows to control better
the numerical implementation.

\section{A recurrence for the computation and estimation of moments.}

The key to our subsequent investigations is the
following:
\begin{lemma}\label{lem2}
  Let $f$ be a bounded function on $[0,b)$.  Then
  \begin{equation*}
    \int_{[0,1)} f(bx)d\mu(x) = f(0) + 
             \int_{[0,1)} \frac1b \sum_{a\in A} f(a+x)d\mu(x)\;.
  \end{equation*}
\end{lemma}
\begin{proof}
  The integrals are absolutely convergent because we proved already that
  $\mu([0,1)) = u_0 < \infty$.  Let us first suppose that $0\in A$.  Then the
  left-hand-side (L.H.S). is, from the definition of $\mu$, $\sum_{l\geq 0} \sum_{n\in
    \mathcal{A}, l(n)\leq l} b^{-l}f(n \cdot b^{1-l})$. Indeed the condition
  $l(n)\leq l$ is equivalent to $0\leq n\cdot b^{-l} < 1$.  From $l=0$ we get a
  term $f(0)$. For $l\geq1$, $l(n)\leq l$, $n \cdot b^{1-l}$ is uniquely
  expressible as $a + m \cdot b^{1-l}$ with $a\in A$ (inclusive of zero) and
  $m$ admissible, $m<b^{l-1}$ (so $m=0$ if $l=1$).  This gives the
  right-hand-side as
  each $a \in A$ contributes a sum indexed by $l-1\geq0$ and admissible
  integers $<b^{l-1}$.

  If $0\notin A$, the L.H.S. is (using \eqref{eq:mu0notinA}) $f(0)+\sum_{l\geq
    1} \sum_{n\in \mathcal{A}, b^{l-1}\leq n< b^{l}} b^{-l}f(n \cdot
  b^{1-l})$.  For $l\geq2$, and $n$ admissible of length $l$, we write
  uniquely $n/b^{1-l} = a + m/b^{l-1}$ with $a\in A$ and $m$ admissible of
  exact length $l-1$.  For fixed digit $a$ the sum over $l\geq2$ will give us
  $b^{-1}\int_{(0,1)}f(a+x)d\mu(x)$.  An additional $b^{-1} f(a)$ is
  contributed by the $l=1$ part with $n=a$ in the original sum thus allowing
  to extend the integration range in that integral from $(0,1)$ to $[0,1)$.
\end{proof}

We shall need the moments $\gamma_j = \sum_{a\in A} a^j$ of $\sum_{a\in A} \delta_a$.
\begin{proposition}\label{prop:recursion}
  The following linear recursion holds:
  \begin{equation*}
    (b^{m+1} - N) u_m = \sum_{j=1}^{m} \binom{m}{j} \gamma_j u_{m-j}\;,\qquad(m\geq1),
  \end{equation*}
  with 
  $u_0 = \frac{b}{b-N}$. In particular $u_1 = \frac{\sum_{a\in A}
    a}{b^2-N}\frac{b}{b-N}$ holds.
\end{proposition}
\begin{proof}
  The value of $u_0$ was already indicated previously.  The value of $u_1$
  follows from the recursion. We apply Lemma \ref{lem2} to the power
  function and obtain, for $m\geq1$:
  \begin{align*}
    b^{m+1} u_m &= \int_{[0,1)} \sum_{a\in A} (a + x)^m \,d\mu(x)\\
               &= \sum_{j=0}^{m} \binom{m}{j} u_{m-j}\sum_{a\in A} a^j\\
               &= N u_m + \sum_{j=1}^{m}  \binom{m}{j}u_{m-j}\gamma_j\;,
  \end{align*}
  which gives the wished-for formula.
\end{proof}
\begin{remark}\label{rem:exp}
  Applying Lemma \ref{lem2} to the function $f(x) = e^{tx}$ we obtain the
  following functional equation for the exponential generating function $E(t)
  = \int_{[0,1)} e^{tx}\,d\mu(x)$:
  \begin{equation*}
    E(bt) = 1 +  \frac1b(\sum_{a\in A} e^{at}) E(t)\;.
\end{equation*}
  It is equivalent to the recursion from Proposition \ref{prop:recursion}.
\end{remark}

\begin{proposition}\label{prop:estimate}
  Let $f=\max A$.
  Let $\lambda_m$ for $m\geq1$ be defined as:
  \begin{equation}\label{eq:lambdam}
    \lambda_m = m \Bigl(\frac{b-1}f\Bigr)^{m-1} \frac{u_{m-1}}{u_0}\;.
  \end{equation}
  They satisfy for $m\geq2$,
  \begin{equation}
    (b^{m} - N) \lambda_{m}  =
    \sum_{j=1}^{m-1} \binom{m}{j}\frac{\gamma_j}{f^j}(b-1)^j\lambda_{m-j}\;,
  \end{equation}
  and the bounds $b^{-1} < \lambda_m < 1$ hold (but
  $\lambda_1=1$).
\end{proposition}
\begin{proof}
  The recursion formula for $\lambda_{m+1}$ follows readily from the one
  of $u_m$.  One only needs the identity
  \begin{equation*}
    \frac{m+1}{m+1-j}\binom{m}{j} = \binom{m+1}{j}\;,
  \end{equation*}
  and then an index replacement $m+1\rightarrow m$ to go back from
  $\lambda_{m+1}$ to $\lambda_m$.

Let $m\geq2$ and suppose $\lambda_j\leq 1$
  is known for $1\leq j<m$ (which is true for $m=2$). We then obtain
\begin{equation*}
    (b^{m} - N) \lambda_{m}
    \begin{aligned}[t]
&\leq \sum_{j=1}^{m-1} \binom{m}{j}\gamma_j(\frac{b-1}{f})^j \\
&= \sum_{a\in
        A_1} \left((\frac{a(b-1)}{f} + 1)^m - \frac{a^m(b-1)^m}{f^m} -
        1\right)\;.
    \end{aligned}
\end{equation*}
The sum in this last equation can be extended from $A_1$ to $A_1\cup\{0\}$,
even if $0\notin A$, with no change, as $m\geq1$.  We do so; hence, we have
$N_1 +1$ contributions.

Let us compare for $a\in A_1$ the term
$-(a(b-1)/f)^m$ with $(a'(b-1)/f +1)^m$ which is contributed by the largest
admissible digit $a'<a$, or by $a'=0$ if $a=\min A_1$ (we leave aside the $-1$'s for the moment).  It is easily checked that
$a'(b-1)/f + 1 \leq a(b-1)/f$ with equality only if $a'=a-1$ and $f=b-1$.
Proceeding down from $a=f$ we see that we can obtain $b^m$ as final result
only if at all stages there was exact compensation.  This happens if and only
if $f=b-1$ and for each $a\in A_1$, $a>1$ one has $a-1 \in A_1$.  This means
that all non-zero digits are in $A$.  Hence, in all other cases
we have a total contribution less than $b^m$.  With the $-1$'s now counted in,
we obtain $(b^m-N)\lambda_m<b^m - N_1 - 1$, hence $\lambda_m < 1$.

Let us return to the exceptional case where all positive digits are admissible
(so $0$ is not admissible). In particular $N = N_1=b-1$.  Then our sum gives
$b^m$ and we had left aside $N_1+1 = N+1 = b$ negative units.  So $(b^m -
(b-1))\lambda_m\leq b^m - b$ and we also obtain in this exceptional case
$\lambda_m<1$.

Hence the recurrence hypothesis $\lambda_j\leq1$ is verified for $j=m$ in the
stronger form $\lambda_m<1$, $m\geq2$.

For the lower bound, as $\gamma_j\geq f^j$, and $N\geq1$, we have $\lambda_m\geq
\xi_m$ for all $m$ if we set $\xi_1=1$ and define $\xi_m$ for $m\geq2$ by the recurrence
  \begin{equation*}
    (b^{m} - 1) \xi_{m}  = \sum_{j=1}^{m-1} \binom{m}{j}(b-1)^j\xi_{m-j} \;.
  \end{equation*}
  These $\xi_m$'s are the $\lambda_m$'s for the case of $A$ being a singleton,
  i.e., $N=1$. We pick for concreteness $A=\{b-1\}$. In that case the
  admissible integers are the $b^l -1$ for $l\geq1$.  The part of the measure
  $\mu$ inside the open unit interval has its support on the $1 - b^{-l}$'s
  with respective weights the $b^{-l}$. The moments $\theta_m$ are thus
  $\theta_0= b/(b-1)$ and for $m\geq1$: $\theta_m = \sum_{l=1}^\infty b^{-l}
  (1-b^{-l})^m$.  Formula \eqref{eq:lambdam} gives for $m\geq1$:
\begin{equation*}
   \xi_{m+1} = (m+1)\frac{\theta_m}{\theta_0} = (m+1)(1 -
   b^{-1})\sum_{l=1}^\infty b^{-l} (1-b^{-l})^m \;.
\end{equation*}
Let $t_0=0$ and $t_l = 1 - b^{-l}$ for $l\geq1$. For $m\geq1$ we have
\begin{equation*}
  \int_0^1 t^m\, dt = \sum_{l=1}^\infty \int_{t_{l-1}}^{t_l} t^m\,dt <
  \sum_{l=1}^\infty (t_l - t_{l-1}) t_l^m=
  (1 - b^{-1})b\sum_{l=1}^\infty b^{-l} (1-b^{-l})^m\;,
\end{equation*}
and we thus obtain $\xi_{m+1} > b^{-1}$, hence $\lambda_{m+1}>b^{-1}$.
\end{proof}

We are now in a position to complete the proof of our main theorem.
\begin{proof}[Proof of \textup{Theorem \ref{thm:main}}]
  Most has already been established so far: equation \eqref{eq:main} is, using
  $u_0=b/(b-N)$, the equation \eqref{eq:main0} from Theorem \ref{thm:1}.  The
  recursion \eqref{eq:recur} is the statement from Proposition
  \ref{prop:recursion}.  The lower and upper bound for $\lambda_m$ were done
  in Proposition \ref{prop:estimate}.

  The geometric convergence for $\ell\geq2$ is a corollary to $\lambda_m\leq1$
  and $\beta_{\ell,m+1}=O((1/b^{\ell-1})^m)$ with $b^{\ell-1}\geq b>1$.  If $\ell=1$ we
  still get geometric convergence from the $\beta_{\ell,m+1}$ if $1\notin A$ as
  then $\beta_{1,m+1}=O(1/2^m)$.
  And we get geometric convergence from the $u_m$ if $f < b-1$.  So the only
  way not to have geometric convergence in equation \eqref{eq:main} is that
  $\ell=1$, $1\in A$, $b-1\in A$.

  There only remains to prove that the sequence $(\lambda_m/m)$ is decreasing.
  If $f = \max A = b-1$ this follows from the analogous fact for $(u_m)$ as it
  then differs from $(\lambda_{m+1}/(m+1))$ only by some multiple.  Our
  measure $\mu$ is composed of weighted Dirac deltas at the $n/b^\ell$'s for
  $n$ admissible.  In the interval $[0,1)$, the supremum of those fractions is
  $f/(b-1)$.  So let us imagine we scale all those fractions by $(b-1)/f$ so
  that their (unattained) supremum becomes $1$, keeping the same weights.  We
  see that the $\lambda_{m+1}/(m+1)$, $m\geq0$ are up to common multiple the
  moments of this new measure on $[0,1)$.  The result follows.
\end{proof}

\section{Related works, generalizations and future directions.}

For the special case of $N=1$, the methods of this paper do not have much
numerical interest it seems (but the
quantities $u_m$'s do have quite interesting properties), as
$\sum_{n\geq1} (b^n-1)^{-1}$ is already geometrically convergent, and even
better admits via the Clausen identity \cite{clausen1828} a very efficient
representation (using $q=b^{-1}$):
  \begin{equation*}
    \sum_{n=1}^\infty \frac{q^n}{1-q^n} = \sum_{n=1}^\infty q^{n^2}\frac{1 +
      q^{n}}{1-q^n}\;.
  \end{equation*}
For $b=2$,
  $q=1/2$, the sum is known as the Erdös-Borwein constant.  According to a
  theorem of Borwein \cite{borwein} these sums for $q=b^{-1}$, $b>1$ integer,
  are irrational numbers.  Analogous results for $N>1$ are unknown to this author.

Fischer \cite{fischer}
obtained the following representation of the value $K$ of the original Kempner
series excluding nine's from the decimal digits:
\begin{equation}\label{eq:fischer}
  K = \beta_0 \log 10 - \sum_{m=2}^\infty 10^{-m} \beta_{m-1} \zeta(m)\;,
\end{equation}
where $\beta_0=10$ and the $\beta_m$'s satisfy the linear recurrence:
\begin{equation}\label{eq:fischerrec}
  \sum_{k=1}^m \binom{m}{k}(10^{m-k+1}-10^k + 1)\beta_{m-k} = 10(11^m -10^m)\;.
\end{equation}
During the write-up of the first version of this paper, the author became
aware of these formulas of Fischer and quoted them as somewhat similar to
Theorem \ref{thm:1}.  It is much more than a similarity: there is actually an
intimate link with the present work. Indeed the Fischer coefficients $\beta_m$
are the ``complementary moments'' $\int_{[0,1)}(1-x)^m\,d\mu(x)$ on the unit
interval of the measure $\mu$ which is at the core of the approach of the
present contribution; $\mu$ defines a linear functional on continuous
functions on the unit interval which one can prove is the one at the heart of
the work of Fischer.

Applying the same method as was done in Proposition \ref{prop:recursion} (or
using Remark \ref{rem:exp}) gives a recurrence formula for the
$\beta_m$'s.  This recurrence formula is not immediately identical with
\eqref{eq:fischerrec}, but implies it after additional manipulations.
The details of these steps are a bit lengthy to reproduce here, see
\cite{burnoldigamma}.
There, the Fischer formulas
\eqref{eq:fischer} and \eqref{eq:fischerrec} have been extended to the full
generality of the present article.  In particular the main term of
\eqref{eq:fischer}, $10\log(10)$, is replaced by $b\log(b)/(b-N)$,
so $b\log(b)$ when exactly one digit $d$ is excluded.  The cases $d=0$ and
$d=b-1$ are the more favorable ones regarding the rapidity of convergence of the
``corrective'' series generalizing \eqref{eq:fischer}.  The author has
obtained asymptotic expansions for $b\to\infty$ with fixed $d$, or with
$d=b-1$ in further works (\cite{burnollargeb}, \cite{burnolasymptotic}).

In another direction, the methods leading to Theorem \ref{thm:main} have been
extended by the author to Irwin series (\cite{burnolirwin}) and to the case of
a string of two consecutive digits (\cite{burnolone42}) which either is excluded or
is authorized exactly once.  The Allouche-Hu-Morin theorem about the behavior
when the number of occurrences of a given string goes to infinity was given a
completely different proof by the author (\cite{burnolblocks}).

\paragraph{\textbf{Acknowledgements.}}
The author thanks the referees and the AMM editors for their remarks which have been of a great help in
improving the presentation of the results.

\singlespacing
\small

\providecommand\bibcommenthead{}
\def\blocation#1{\unskip}


\end{document}